\documentclass[a4paper,11pt]{article}
\usepackage{amsmath}
\usepackage{amsfonts}
\usepackage{amssymb}
\usepackage{amsthm}
 \begin{document}
\title{QUASI VALUATION AND VALUATION DERIVED FROM FILTERED RING AND THEIR PROPERTIES}
\author{M.H. Anjom SHoa\footnote{Mohammad Hassan Anjom SHoa, University of Biirjand, anjomshoamh@birjand.ac.ir}, M.H. Hosseini\footnote{Mohammad Hossein Hosseini, University of Birjand, mhhosseini@birjand.ac.ir}}
\date{}
\maketitle
\newtheorem{Def}{Definition}[section]
\newtheorem{lem}{Lemma}[section]
\newtheorem{theo}{Theorem}[section]
\newtheorem{pro}{Proposition}[section]
\newtheorem{cro}{Corollary}[section]
\newtheorem{rem}{Remark}[section]
\begin{abstract}
In this paper we show if $R$ is a filtered ring then we can define a quasi valuation. And if  $R$ is some kind of  filtered ring then we can define a valuation. Then we prove some properties and relations for $R$.
\end{abstract}

\section{Introduction}
 \paragraph*{}In algebra valuation ring and filtered ring are two most important structure \cite{5},\cite{6},\cite{7}. We know that filtered ring is also the most important structure since filtered ring is a base for graded ring especially associated graded ring and completion and some similar results, on the Andreadakis--Johnson filtration of the automorphism  group of a free group \cite{1}, on the depth of the associated graded ring of a filtration \cite{2},\cite{3}. So, as these important structures, the relation between these structure is useful for  finding some new structures, and if $R$ is a discrete valuation ring then  $R$ has many properties that have many usage for example Decidability of the theory of modules over commutative valuation domains \cite{7},  Rees valuations and asymptotic primes of rational powers in Noetherian rings and lattices
\cite{6}.
 \paragraph*{} In this article we investigate the relation between filtered ring and valuation and quasi valuation ring. We prove that if we have filtered ring then we can find a quasi valuation on it. Continuously we show that if $R$ be a strongly filtered then exist a valuation, Similarly we prove it for
 PID. At the end we explain some properties for them.

\section{Preliminaries}
\begin{Def}
  A filtered ring $R$ is a ring together with a family $\left\{R_{n} \right\}_{n\ge 0} $ of  subgroups of $R$ satisfying in the following conditions
  \begin{enumerate}
  \item[i)] $R_{0} =R$;
  \item[ii)]$R_{n+1} \subseteq R_{n} $ for all $n\ge 0$;
  \item[iii)]$R_{n} R_{m} \subseteq R_{n+m} $ for all  $n,m\ge 0$.
    \end{enumerate}
 \end{Def}

  \begin{Def}
  Let $R$ be a ring together with a family $\left\{R_{n} \right\}_{n\ge 0} $ of  subgroups of $R$ satisfying the following conditions:
  \begin{enumerate}

  \item[i)]$R_{0} =R$;
  \item[ii)] $R_{n+1} \subseteq R_{n} $ for all $n\ge 0$;
  \item[iii)]$R_{n} R_{m} =R_{n+m} $  for all  $n,m\ge 0$,

Then we say  $R$ has a strong filtration.
  \end{enumerate}
 \end{Def}
 \begin{Def}
 Let $R$ be a ring and $I$ an ideal of $R$. Then $R_{n}=I^{n}$ is called $I$-adic filtration.
 \end{Def}
\begin{Def}
 A map $f:M\to N$  is called  a homomorphism of filtered modules if: (i) $f$  is $R$-module  an  homomorphism  and (ii) $f(M_{n} )\subseteq N_{n} $ for all $n\ge 0$.
 \end{Def}
 \begin{Def}
 A subring $R$ of a filed $K$  is called a valuation ring of $K$ if for every $\alpha \in K$ , $\alpha \ne 0$ , either $\alpha \in R$ or $\alpha ^{-1} \in R$.
 \end{Def}
  \begin{Def}
 Let $\Delta $ be a totally ordered abelian group. A valuation $\nu $ on $R$ with  values  in $\Delta $  is a mapping $\nu :R^{*} \to \Delta $  satisfying :
 \begin{itemize}
 \item[i)]$\nu (ab)=v(a)+v(b)$;
 \item[ii)]$v(a+b)\ge Min\left\{v(a),v(b)\right\}$.
 \end{itemize}
 \end{Def}
 \begin{Def}\label{sec:deff123}
 Let $\Delta $ be a totally ordered abelian group. A quasi valuation $\nu $ on $R$ with  values  in $\Delta $  is a mapping $\nu :R^{*} \to \Delta $  satisfying :
 \begin{itemize}
 \item[i)] $\nu (ab)\ge v(a)+v(b)$;
 \item [ii)]$v(a+b)\ge Min\left\{v(a),v(b)\right\}$.
 \end{itemize}
 \end{Def}
 \begin{rem}
 $R$ is said to be \textbf{vaulted ring}; $R_\nu=\lbrace x \in R : \nu(x)\geq0\rbrace$ and $\nu^{-1}(\infty)=\lbrace x \in R : \nu(x)=\infty\}$.
 \end{rem}
 \begin{Def}
 Let $K$ be a filed. A discrete valuation on  $K$ is a valuation $\nu :K^{*} \to \mathbb{Z}$ which is surjective.
 \end{Def}
  \begin{theo}\label{sec:the2}
 If $R$ is a $UFD$ then $R$ is a $PID$ (see [2]).
 \end{theo}
 \begin{pro} \label{sec:po2}
 Any discrete valuation ring is a Euclidean domain(see[3]).
 \end{pro}
 \begin{rem}
 If $R$ is a ring, we will denote by $Z(R)$ the set of \textbf{zero-divisors} of $R$ and by $T(R)$ the \textbf{total ring of fractions}  of $R$.
 \end{rem}
 \begin{Def}\label{sec:madef1}
 A ring $R$ is said to be a \textbf{Manis valuation ring} (or simply a \textbf{Manis ring}) if there exist a valuation $\nu$ on its total fractions $T(R)$, such that $R=R_\nu$.
\end{Def}
\begin{Def}
A ring $R$ is said to be a \textbf{Pr\"{u}fer ring} if each
overring of $R$ is integrally closed in $T(R)$.
\end{Def}
\begin{Def}\label{sec:madef2}
A Manis ring $R_\nu$ is said to be \textbf{$\nu$-closed} if
$R_\nu/\nu^{-1}(\infty)$ is a valuation domain (see Theorem 2 of
\cite{8}).
\end{Def}

 %%\end{Def}
 \section{ Quasi Valuation and Valuation derived from Filtered ring}
Let $R$ be a ring with unit and $R$ a filtered ring with filtration $\left\{R_{n} \right\}_{n>0} $.
 \begin{lem}\label{sec:lem91}
 Let $R$ be a filtered ring  with filtration $\left\{R_{n} \right\}_{n>0} $. Now we define $\nu :R\to\mathbb{Z}$ such that  for every $\alpha \in R$ and $\nu (\alpha )=\min \left\{i\left|\alpha \in R_{i} \backslash R_{h+1} \right. \right\}$. Then we have $v(\alpha \beta )\ge v(\alpha )+v(\beta )$.
  \end{lem}
 \begin{proof}
  For any  $ \alpha ,\beta \in R$  with $\nu (\alpha )=i$ and $\nu (\beta )=j$ , $\alpha \beta \in R_{i} R_{j} \subseteq R_{i+j} $.
  Now let $v(\alpha \beta )=k$ then we have $\alpha \beta \in R_{k} \backslash R_{k+1} $ .\\We show that $k\ge i+j$.\\
  Let $k<i+j$ so we have $k+1\le i+j$ hence $R_{K+1} \supset R_{i+j} $ then  $\alpha \beta \in R_{i+j} \subseteq R_{k+1} $ it is contradiction. So $k\ge i+j$. Now we have $v(\alpha \beta )\ge v(\alpha )+v(\beta )$.\\
 \end{proof}
 \begin{lem}\label{sec:lem92}
  Let $R$ be a filtered ring  with filtration $\left\{R_{n} \right\}_{n>0} $. Now we define $\nu :R\to \mathbb{Z}$ such that  for every $\alpha \in R$ and $\nu (\alpha )=\min \left\{i\left|\alpha \in R_{i} \backslash R_{h+1} \right. \right\}$. Then $v(\alpha +\beta )\ge \min \left\{v(\alpha ),v(\beta )\right\}$.
 \end{lem}
 \begin{proof}
 For any $\alpha ,\beta \in R$  such that $\nu (\alpha )=i$ also$\nu (\beta )=j$ and $v(\alpha +\beta )=h$ so we have $\alpha +\beta \in R_{k} \backslash R_{k+1} $. Without losing the generality, let $i<j$ so $R_{j} \subset R_{i} $ hence $\beta \in R_{i} $. Now if  $k<i$ then  $k+1\le i$ and $R_{i} \subset R_{k+1} $ so $\alpha +\beta \in R_{i} \subset R_{k+1} $ it is contradiction. Hence$k\ge i$ and so we have $v(\alpha +\beta )\ge \min \left\{v(\alpha ),v(\beta )\right\}$.\\
 \end{proof}
 \begin{theo}\label{sec:th1}
  Let $R$ be a filtered ring. Then there exist a quasi valuation on $R$.
 \end{theo}
 \begin{proof}
 Let $R$ be a filtered ring  with filtration $\left\{R_{n} \right\}_{n>0} $. Now we define $\nu :R\to \mathbb{Z}$ such that  for every $\alpha \in R$ and $\nu (\alpha )=\min \left\{i\left|\alpha \in R_{i} \backslash R_{h+1} \right. \right\}$. Then \item[i)] By lemma(\ref{sec:lem91}) we have $v(\alpha \beta )\ge v(\alpha )+v(\beta )$ .
 \item[ii)] By lemma(\ref{sec:lem92}) we have  $v(\alpha +\beta )\ge \min \left\{v(\alpha ),v(\beta )\right\}$.
 So by \ref{sec:deff123} $R$ is quasi valuation ring.

 \end{proof}
 \begin{pro}\label{sec:po123}
 Let $R$ be a strongly filtered ring. Then there exists a valuation on $R$.
 \end{pro}
 \begin{proof}
 By theorem  (\ref{sec:th1}) we have  $v(\alpha \beta )\ge v(\alpha )+v(\beta )$  and  $v(\alpha +\beta )\ge \min \left\{v(\alpha ),v(\beta )\right\}$. Now  we show  $v(\alpha \beta )=v(\alpha )+v(\beta )$. Let $v(\alpha \beta )>v(\alpha )+v(\beta )$ so $k>i+j$ and it is contradiction. So  $v(\alpha \beta )=v(\alpha )+v(\beta )$, then  there is a valuation on $R$.
 \end{proof}
 \begin{cro}
Let $R$ be a strongly filtered ring, then $R$ is a Euclidean
domain.
 \end{cro}
 \begin{proof}
  By proposition (\ref{sec:po123}) $R$ is a discrete valuation and so by proposition (\ref{sec:po2}) $R$ is a Euclidean domain.
  \end{proof}
 \begin{pro} \label{sec:po3}
  Let $P$ is a prime ideal of $R$ and $\left\{{P^n}\right\}_{n\ge0}$ be P-adic filtration. Then there exists a valuation on $R$.
 \end{pro}
 \begin{proof}
  By theorem  (\ref{sec:th1}) we have  $v(\alpha \beta )\ge v(\alpha )+v(\beta )$  and  $v(\alpha +\beta )\ge \min \left\{v(\alpha ),v(\beta )\right\}$. Now  we show  $v(\alpha \beta )=v(\alpha )+v(\beta )$. Let $v(\alpha \beta )>v(\alpha )+v(\beta )$ so $k>i+j$  then $\alpha \beta \in P^{k} \mathop{\subset }\limits_{\rlap{$/$}-} P^{i+j} $ and  $k\ge i+j+1$, since  $P$ is a prime ideal hence $\alpha \in P^{i+1} $ or $\beta \in P^{j+1} $ and it is contradiction. So  $v(\alpha \beta )=v(\alpha )+v(\beta )$, then  there is a valuation on $R$.
 \end{proof}
 \begin{pro} \label{sec:po1}
  Let  $R$ be a $PID$ then there is a valuation on $R$.
 \end{pro}
 \begin{proof}
 By theorem (\ref{sec:th1}) and proposition (\ref{sec:po3}) there is a valuation on $R$.
 \end{proof}
 \begin{cro}
 If  $R$ is an $UFD$ then there exists a valuation on $R$, then $R$ is a Euclidean domain.
 \end{cro}
  \begin{cro}
  Let  $R$ be a ring  and $P$ is a prime ideal of $R$. If $R$ has a $P \_\ adic$ filtration and $R=\bigcup _{i=0}^{+\infty }P^{i}  $, then $R$ is a Euclidean domain.
  \end{cro}
  \begin{proof}
  By proposition (\ref{sec:po3}) $R$ is a discrete valuation and so by proposition (\ref{sec:po2}) $R$ is a Euclidean domain.
  \end{proof}
  \begin{cro}
   Let  $R$ be a $PID$ then $R$ is a Euclidean domain.
  \end{cro}
  \begin{proof}
  By proposition (\ref{sec:po1}) and proposition (\ref{sec:po2}) we have $R$ is a Euclidean domain.
  \end{proof}
  \begin{cro}
  Let  $R$ be a $UFD$ then $R$ is a Euclidean domain.
  \end{cro}
  \begin{cro}
Let $R$ be a strongly filtered ring. Then $R$ is Manis ring.
  \end{cro}
  \begin{cro}\label{sec:mco1}
Let $P$ is a prime ideal of $R$ and
$\left\{{P^n}\right\}_{n\ge0}$ be P-adic filtration. Then $R$ is
Manis ring
  \end{cro}
  \begin{pro}\label{sec:mpro1}
  Let $R_\nu$ be a Manis ring. If $R_\nu$ is $\nu$-closed, then
  $R_\nu$ is Pr\"{u}fer.
  \end{pro}
  \begin{proof}
  See proposition 1 of \cite{9}
  \end{proof}
  \begin{pro}\label{sec:mpros1}
Let $R$ be a strongly filtered ring. Then $R$ is $\nu$-closed.
  \end{pro}
  \begin{proof}
  By proposition (\ref{sec:po123}) and definition (\ref{sec:madef1})
we have $R$ is Manis ring and $R=R_\nu$.\\ Now let $\alpha,\beta
\in R$ and
\begin{equation*}
\nu(\alpha)=i \-\ and \-\ \nu(\beta)=j %\-\ where \-\ i,j\neq\infty%
\end{equation*}
Consequently if
\begin{equation*}
(\alpha+\nu^{-1}(\infty))(\beta+\nu^{-1}(\infty))\in
\nu^{-1}(\infty)
\end{equation*}
Then $i+j\geq\infty$ so $\alpha\in\nu^{-1}(\infty)$ or
$\beta\in\nu^{-1}(\infty)$.Hence by definition (\ref{sec:madef2})
$R$ is $\nu$-closed.
\end{proof}
\begin{cro}
 Let $R$ be a strongly filtered ring. Then $R$ is
Pr\"{u}fer.
\end{cro}
\begin{proof}
By proposition (\ref{sec:mpro2}) $R$ is $\nu$-closed so by
proposition (\ref{sec:mpro1}) $R$ is Pr\"{u}fer.
\end{proof}
  \begin{pro}\label{sec:mpro2}
  Let $P$ is a prime ideal of $R$ and
$\left\{{P^n}\right\}_{n\ge0}$ be P-adic filtration. Then $R$ is
$\nu$-closed.
\end{pro}
\begin{proof}
By proposition (\ref{sec:po3}) and definition (\ref{sec:madef1})
we have $R$ is Manis ring and $R=R_\nu$.\\ Now let $\alpha,\beta
\in R$ and
\begin{equation*}
\nu(\alpha)=i \-\ and \-\ \nu(\beta)=j %\-\ where \-\ i,j\neq\infty%
\end{equation*}
Consequently if
\begin{equation*}
(\alpha+\nu^{-1}(\infty))(\beta+\nu^{-1}(\infty))\in
\nu^{-1}(\infty)
\end{equation*}
Then $i+j\geq\infty$ so $\alpha\in\nu^{-1}(\infty)$ or
$\beta\in\nu^{-1}(\infty)$.Hence by definition (\ref{sec:madef2})
$R$ is $\nu$-closed.
\end{proof}
\begin{cro}
  Let $P$ is a prime ideal of $R$ and
$\left\{{P^n}\right\}_{n\ge0}$ be P-adic filtration. Then $R$ is
Pr\"{u}fer.
\end{cro}
\begin{proof}
By proposition (\ref{sec:mpro2}) $R$ is $\nu$-closed so by
proposition (\ref{sec:mpro1}) $R$ is Pr\"{u}fer.
\end{proof}

%% \begin{Def}
%% \end{Def}
\end{document}